\documentclass[11pt]{article}
\usepackage[a4paper,margin=1.25in]{geometry}
\usepackage[colorlinks,citecolor=magenta,linkcolor=black]{hyperref}
\pdfpagewidth=\paperwidth \pdfpageheight=\paperheight
\usepackage{amsfonts,amssymb,amsthm,amsmath,eucal,tabu,url}
\usepackage{pgf}
 \usepackage{array}
 \usepackage{tikz-cd}
 \usepackage{pstricks}
 \usepackage{pstricks-add}
 \usepackage{pgf,tikz}
 \usetikzlibrary{automata}
 \usetikzlibrary{arrows}
 \usepackage{indentfirst}
\usepackage{tabularx} 

\theoremstyle{plain}
\newtheorem{thm}{Theorem}[section]
\newtheorem{theorem}[thm]{Theorem}

\newtheorem{proposition}[thm]{Proposition}
\newtheorem{corollary}[thm]{Corollary}

\theoremstyle{definition}
\newtheorem{definition}[thm]{Definition}
\newtheorem{remark}[thm]{Remark}
\newtheorem{example}[thm]{Example}

\newtheorem{question}[thm]{Question}

\newtheorem{thevarthm}[thm]{\varthmname}

\newenvironment{varthm*}[1]{\trivlist\item[]{\bf #1.}\it}{\endtrivlist}

\renewcommand\geq{\geqslant}

\renewcommand\leq{\leqslant}

\newcommand\be{\begin{eqnarray*}}
\newcommand\ee{\end{eqnarray*}}

\newcommand\calo{{\mathcal O}}

\newcommand\newop[2]{\def#1{\mathop{\rm #2}\nolimits}}
\newop\log{log}
\newop\ord{ord}
\newop\Gal{Gal}
\newop\SL{SL}
\newop\Bl{Bl}
\newop\mult{mult}
\newop\mass{mass}
\newop\div{div}
\newop\codim{codim}
\newop\sing{sing}
\newop\vdim{vdim}
\newop\edim{edim}
\newop\Ass{Ass}
\newop\size{size}
\newop\reg{reg}
\newop\satdeg{satdeg}
\newop\supp{supp}
\newop\Neg{Neg}
\newop\Nef{Nef}
\newop\Nefh{Nef_H}
\newop\Eff{Eff}
\newop\Zar{Zar}
\newop\MB{MB}
\newop\MBxC{MB\mathit{(x,C)}}
\newop\NnB{NnB}
\newop\Bigg{Big}
\newop\Effbar{\overline{\Eff}}

\def\keywordname{{\bfseries Keywords}}%
\def\keywords#1{\par\addvspace\medskipamount{\rightskip=0pt plus1cm
\def\and{\ifhmode\unskip\nobreak\fi\ $\cdot$
}\noindent\keywordname\enspace\ignorespaces#1\par}}
\def\subclassname{{\bfseries Mathematics Subject Classification
(2020)}\enspace}
\def\subclass#1{\par\addvspace\medskipamount{\rightskip=0pt plus1cm
\def\and{\ifhmode\unskip\nobreak\fi\ $\cdot$
}\noindent\subclassname\ignorespaces#1\par}}

\begin{document}
\title{Maximizing curves viewed as free curves}
\author{Alexandru Dimca and Piotr Pokora}
\date{\today}
\maketitle
\thispagestyle{empty}
\begin{abstract}
The aim of this paper is to provide a direct link between 
maximizing curves that occur in the construction of smooth algebraic surfaces having the maximal possible Picard numbers and reduced free plane curves with simple singularities. We also investigate odd degree plane curves with simple singularities having maximal total Tjurina number.
\subclass{14H50, 14N20, 32S22}
\end{abstract}
\section{Introduction}

The study of complex smooth projective surfaces with a maximum Picard number is a classical topic in algebraic geometry, see for example
\cite{Beau, Persson}. One way to construct such surfaces is to start with a plane curve $C \subset \mathbb{P}^{2}_{\mathbb{C}}$ of even degree with only ${\rm ADE}$ singularities, take a double cover $X$ of $\mathbb{P}^{2}_{\mathbb{C}}$ branched along $C$, and study
$\widetilde{X}$ the minimal resolution of singularities of $X$. U. Persson found a sufficient condition on $C$ such that the corresponding surface $\widetilde{X}$ has a maximum possible Picard number and called such curves {\it maximazing}, see
\cite{Persson}. He also gave many very interesting examples of such curves in this paper.

The free divisors were introduced by K. Saito in \cite{KS} and have been the centre of much interest since then, mainly because of Terao's conjecture that the freeness of a hyperplane array is determined by the combinatorics, see \cite{Te}. This conjecture is still open even in the case of line arrangements in $\mathbb{P}^{2}_{\mathbb{C}}$, despite a lot of work in the last 40 years, see \cite{Bar} for an update of the current situation. It is natural to ask whether Terao's conjecture can somehow be generalised to all reduced plane curves, and as it turned out in \cite{SchenckToh}, such an expectation was too optimistic. In order to better understand the freeness of reduced plane curves, one tries to find certain classes of reduced curves with some extreme properties and check whether they are free, and this strategy has been explored e.g. in \cite{DJP, DimcaPokora, PSz}.

In particular, the free plane curves in $\mathbb{P}^{2}_{\mathbb{C}}$, to be defined below in Definition \ref{defF} and whose exponents are introduced in Definition \ref{defEXP}, are characterized by a certain maximality property of their total Tjurina numbers, see \cite{Dimca1,duP}, which we recall below in Theorem \ref{thm1}. Therefore it is natural to ask if there is a relation between maximizing curves in the Persson sense and free curves with only simple singularities. Our main result is that a maximizing curve of even degree $n=2m$ is exactly a free curve with the exponents $(m-1,m)$ withonly ${\rm ADE}$ singularities, see Theorem \ref{maxi}. This result is proved in Section 2, after recalling all the necessary notions and results for its proof are recalled.

In Section 3 we present some of the examples of maximizing curves given by Persson in \cite{Persson}, which in view of our Theorem \ref{maxi} gives new interesting examples of free curves.
In Example \ref{P5} we consider curves of arbitrary even degree $n=2m+2$, consisting of 2 lines and one remaining curve, denoted by $\mathcal D_{2m}$ in Remark \ref{rkNF}, which has been considered by several authors, see for instance \cite{Artal}. This curve is not free, but has a certain maximality property of its total Tjurina number, which makes it a nearly free curve in the sense of \cite{Artal,DimStic}, see Remark \ref{rkNF}. Moreover, this curve $\mathcal D_{2m}$ has other interesting properties, giving examples of curves with a maximal number of $A_5$ (resp. $A_7$) singularities, see Example \ref{exP1} (resp. Example \ref{exP2}) in the fourth section.
 
 In Section 5 we study reduced plane curves $C$ of {\it odd degree} $n=2m+1$ with at most ${\rm ADE}$ singularities and show that their total Tjurina number $\tau(C)$ is at most $3m^2+1$. When the equality holds, then we call such $C$ a maximizing curve, and note that this happens if and only if $C$ is a free curve with the exponents $(m-1,m+1)$, see Proposition \ref{propODD1}. Then we construct examples of maximizing curves of degree $5$ and $7$, and ask about the existence of such curves in the odd degrees $\geq 9$. Even if odd degree curves cannot be used to construct interesting surfaces by taking the associated double covers, see also Remark \ref{rkRODD} for odd degree covers, they have a {\it geometric interest}.  This is due to the fact that, by adding a well-chosen line to a maximizing curve of odd degree, one can obtain a maximizing curve of even degree, see for instance Example \ref{exEVEN0}. We present results devoted to our general construction where we show that adding a new irreducible component $C_2$ to a curve $C_1$ with a high total Tjurina number can lead to a maximizing curve $C=C_1 \cup C_2$ - see Theorem \ref{thmODD} and Theorem \ref{thmEVEN} for the case when $C_2$ a line, and Theorem \ref{thmODD2} and Theorem \ref{thmEVEN2} for the case where $C_2$ a smooth conic.

We hope that our note will lead to a renewed interest in the construction of maximizing curves for double covers of the complex projective plane, which will also lead to new examples of reduced free curves with simple singularities.

\section{Maximizing curves as free curves}

It is well-known that a double cover of a non-singular surface is non-singular if and only if the branch curve is non-singular. However, if the singularities of curves are mild, these singularities give rise to mild singularities for the double cover.  In order to have only rational double points as singularties for our surfaces, we focus only on branching curves with ${\rm ADE}$ singularities, and we  call such curves \textbf{simply singular}, following \cite{Persson}. The classification of ${\rm ADE}$ singularities is a very classical subject, see for instance \cite{RCS}, but for the notation reason let us recall it here:
\begin{center}
\begin{tabular}{ll}
$A_{k}$ with $k\geq 1$ & $: \, x^{2}+y^{k+1}  =0$, \\
$D_{k}$ with $k\geq 4$ & $: \, y^{2}x + x^{k-1}  = 0$,\\
$E_{6}$ & $: \, x^{3} + y^{4}=0$, \\
$E_{7}$ & $: \, x^{3} + xy^{3}=0$, \\
$E_{8}$ & $:\, x^{3}+y^{5} = 0$.
\end{tabular}
\end{center}
Recall that a $W_{n}$ singularity of a branch curve $C$ with $W \in \{A,D,E\}$ gives rise to a singularity of the corresponding double cover $X$ whose resolution contributes exactly $n$ cycles to the N\'eron-Severi group of the smooth surface $\widetilde{X}$. Based on that observation, we have the following definition.
\begin{definition}
The index of a singularity $W_{n}$ with $W\in \{A,D,E\}$ is defined as the integer $n$ in the subscript. The index of a simply singular curve $C$ is defined as the sum of the indices over all singular points of the curve, and we denote it by $\sigma(C)$.
\end{definition}
In fact, the index of a singularity is nothing else as the local Milnor number, but here we want to follow notations used by Persson in \cite{Persson}. 
If $C$ is a reduced curve in the complex projective plane $\mathbb{P}^{2}_{\mathbb{C}}$, then we denote by $\tau(C)$ its total Tjurina number, i.e., this is the sum over all singular points of $C$ of the local Tjurina numbers of $C$ at $p$. If necessary, for the definition of local Milnor and Tjurina numbers refer for instance to \cite{RCS}.
The first observation that we would like to indicate here is the following.
\begin{remark}
\label{rktau}
Let $C \subset \mathbb{P}^{2}_{\mathbb{C}}$ be a reduced simply singular curve, i.e., admitting only ${\rm ADE}$ singularities. Then
$$\sigma(C) = \tau(C).$$
This equality comes from the fact that in Arnold's notation for simple singularities, say $W_m$ with $W \in \{A,D,E\}$, the subscript $m$ denotes the Milnor number of the singularity $W_m$. All the simple singularities are quasi-homogeneous, and for such singularities the corresponding Milnor numbers and Tjurina numbers coincide. For these and other basic facts on singularities, one can see for instance \cite{RCS}.
\end{remark}
Now we focus on the discussion related to double covers. \\
Let $C \subset \mathbb{P}^{2}_{\mathbb{C}}$ be a reduced simply singular curve of even degree and consider a double cover $f : X \rightarrow \mathbb{P}^{2}_{\mathbb{C}}$ branched along $C$ and let $\pi : \widetilde{X}\rightarrow X$ be the minimal resolution of singularities of $X$.
We will call $\widetilde{X}$ the surface associated with $C$.
Denote by $\rho(\widetilde{X})$ the Picard number of $\widetilde{X}$.

\begin{proposition}
Let $C \subset \mathbb{P}^{2}_{\mathbb{C}}$ be a reduced simply singular curve of even degree $n$, then we have the following bounds:
$$\sigma(C)+1 \leq \rho(\widetilde{X}) \leq 3\frac{n}{2}\bigg(\frac{n}{2}-1\bigg)+2.$$
\end{proposition}

The proof of this proposition can be found in \cite[Section 4]{Hirz82}, and it is based on several key facts, among which we can find the following.
\begin{enumerate}
\item[a)] The number of exceptional curves in the minimal resolution of a simple singularity $W_m$ is $m$. Hence the union of all such exceptional curves spans a vector subspace of dimension $\sigma(C)$ in the  N\'eron-Severi group ${\rm NS}(\widetilde{X}, \mathbb{R})$, which is negative definite with respect to the cup-product on $H^2(\widetilde{X},\mathbb{R})$.
\item[b)] The (topological) signature theorem implies that the cup-product on the real vector space $H^{1,1}_R:= H^2(\widetilde{X},\mathbb{R}) \cap H^{1,1}(\widetilde{X})$ has signature
$(1,h^{1,1}(\widetilde{X})-1).$
\item[c)] Lefschetz Theorem on $(1,1)$-clases implies that 
${\rm NS}(\widetilde{X}, \mathbb{R})$ can be regarded as a linear subspace of $H^{1,1}_R$
and hence $\sigma(C) \leq h^{1,1}(\widetilde{X})-1$.
\item[d)] Brieskorn Theorem on simultaneous resolution of simple surface singularities implies that the topology and the Hodge numbers of the minimal resolution $\widetilde{X}$ are the same as those of the double covering of  $\mathbb{P}^{2}_{\mathbb{C}}$ ramified along a smooth curve of degree $n$. This allows the computation of
the Hodge numbers, in particular the equality
$$h^{1,1}(\widetilde{X})=3\frac{n}{2}\bigg(\frac{n}{2}-1\bigg)+2.$$
\end{enumerate}
Based on this result, U. Persson has introduced the following definition, for a more general situation when  $\mathbb{P}^{2}_{\mathbb{C}}$ is replaced by any smooth surface, see
\cite[Definition 1.6]{Persson}.
\begin{definition}
\label{defmax}
A reduced simply singular curve $C \subset \mathbb{P}^{2}_{\mathbb{C}}$ of even degree $n\geq 4$ is a \textbf{maximizing curve} if
$$\sigma(C) = 3\frac{n}{2}\bigg(\frac{n}{2}-1\bigg)+1.$$
\end{definition}
For such a maximizing curve $C$ the associated smooth surface $\widetilde{X}$ might have the maximal possible Picard number. From this perspective, it is very interesting to classify reduced simply singular curves that are maximizing. It turns out this problem can be translated into another, namely we can focus on finding reduced simply singular free curves. Now we explain the flip side of our discussion.

Let $S := \mathbb{C}[x,y,z] = \bigoplus_{d}S_{d}$ be the graded polynomial ring. Define ${\rm Der}(S) = \{ \partial := a\cdot \partial_{x} + b\cdot \partial_{y} + c\cdot \partial_{z}, \,\, a,b,c \in S\}$ the free $S$-module of $\mathbb{C}$-linear derivations of the ring $S$.  The module 
${\rm Der}(S)$ becomes a graded $S$-module if we set
\begin{equation} \label{eqDEG} 
\deg (\partial_{x})=\deg (\partial_{y})=\deg (\partial_{z})=0.
\end{equation} 
Now for a reduced curve $C \, : f=0$ given by the defining homogeneous polynomial $f \in S_{d}$, we introduce
$${\rm D}(f) = \{ \partial \in {\rm Der}(S) \, : \, \partial\, f \in \langle f \rangle \}.$$
As we can observe, ${\rm D}(f)$ is the graded $S$-module of derivations preserving the ideal $\langle f \rangle$. For a reduced plane curve $C \, : f=0$ in $\mathbb{P}^{2}_{\mathbb{C}}$, the Euler identity  yields the following decomposition
$${\rm D}(f) = {\rm D}_{0}(f) \oplus S\cdot \delta_{E},$$
where $\delta_{E} = x\partial_{x} + y\partial_{y} + z\partial_{z}$ is the Euler derivation, and 
$${\rm D}_{0}(f) = \{ \partial \in {\rm Der}(S) \, : \, \partial \, f = 0\},$$
is the set of all $\mathbb{C}$-linear derivations of $S$ killing the polynomial $f$.
\begin{definition}
\label{defF} 
We say that a reduced curve $C \, : f=0$ in $\mathbb{P}^{2}_{\mathbb{C}}$ defined by a homogeneous polynomial $f \in S$ is \textbf{free} if ${\rm D}(f)$, or ${\rm D}_{0}(f)$, is a free graded $S$-module.
\end{definition}
The coherent sheaf $E$ on $\mathbb{P}^{2}_{\mathbb{C}}$ associated to the graded $S$-module ${\rm D}_0(f)$ is a rank $2$ vector bundle, called the sheaf of logarithmic vector fields along the curve $C$, see for instance \cite{DimSer}. The curve $C$ is free if and only if this vector bundle splits, that is
$$E=\calo(-d_1)\oplus \calo(-d_2),$$
where $\calo$ is the structure sheaf of $\mathbb{P}^{2}_{\mathbb{C}}$ and 
$(d_1,d_2)$ are the exponents of the free curve $C$, see Definition
\ref{defEXP} below for an alternative definition of exponents.
\begin{definition}
\label{defMDR} 
The minimal degree among derivations killing $f$ is defined as
$${\rm mdr}(f) = {\rm min}\{r \in \mathbb{N} \, : \, {\rm D}_{0}(f)_{r} \neq 0 \}.$$
\end{definition}
In practice, it is rather difficult to check by hand whether a certain reduced plane curve is free. However, we have the following result by du Plessis and Wall \cite{duP}, saying that the total Tjurina number $\tau(C)$ of a curve $C: f=0$ is bounded by an explicit  function $\tau(n,r)_{max}$ which depends only on $n=\deg(C)$ and $r = {\rm mdr}(f)$.

\begin{theorem}
\label{thm1}
Let $C: \,f=0$ be a reduced plane curve of degree $n$ and let $r = {\rm mdr}(f)$.
Then the following two cases hold.
\begin{enumerate}
\item[a)] If $r < n/2$, then $\tau(C) \leq \tau(n,r)_{max}= (n-1)(n-r-1)+r^2$ and the equality holds if and only if the curve $C$ is free.
\item[b)] If $n/2 \leq r \leq n-1$, then
$\tau(C) \leq \tau(n,r)_{max}$,
where, in this case, we set
$$\tau(n,r)_{max}=(n-1)(n-r-1)+r^2- \binom{2r-n+2}{2}.$$
\end{enumerate}
\end{theorem}
When all the singularities of $C:f=0$ are quasi-homogeneous, in particular when $C$ has only simple singularities, it is often easy to compute $\tau(C)$. If we have some indication on $r={\rm mdr}(f)$, in many cases we can conclude that $C$ is a free curve, even without using a computer, as in the proof of Theorem \ref{maxi} below.
\begin{definition}
\label{defEXP} 
The exponents of a reduced free curve $C \subset \mathbb{P}^{2}_{\mathbb{C}}$ of degree $n$ given by $f \in S_{n}$ are defined as the pair of positive integers of the form 
$$(d_1,d_2)=({\rm mdr}(f),n-{\rm mdr}(f)-1).$$
For a free curve, the free graded $S$-module ${\rm D}_{0}(f)$ has rank 2 as an $S$-module, and the exponents are the degrees of a basis for this module, see \cite{ST}. 
\end{definition}
When the curve is not free, the number of generators of the $S$-module
${\rm D}_{0}(f)$ can be quite large, see \cite{DStmax, HS12}. 
Since $\tau(C) = \sigma(C)$, our problem reduces to understand the link between reduced simply singular curves that are free and those that are maximizing. 
The main result of the paper can be formulated as follows.
\begin{theorem}
\label{maxi}
Let $C$ be a plane curve of degree $n=2m \geq 4$ having only ${\rm ADE}$ singularities. Then $C$ is maximizing if and only if $C$ is a free curve with the exponents $(m-1,m)$.
\end{theorem}
\begin{remark}
Notice that, according to Theorem \ref{thm1}, we can call free curves as \textbf{maximal Tjurina curves}. Due to this reason, we arrive at the slogan that the world of maximal Tjurina curves is linked with the world of maximizing curves.
\end{remark}
Before we present our proof, we need an additional preparation. Let us recall that for a reduced curve $C : f=0$ we define the Arnold exponent $\alpha_{C}$ to be the minimum of the Arnold exponents of the singular points $p$ in $C$. Using the modern language, the Arnold exponents of singular points are nothing else than the log canonical thresholds of singularities. For the special case of a quasi-homogeneous singularity, this Arnold exponent can be easily computed as follows.
Recall that the germ $(C,p)$ is quasi-homogeneous of type $(w_{1},w_{2};1)$ with $0 < w_{j} \leq 1/2$ if there are local analytic coordinates $y_{1}, y_{2}$ centered at $p=(0,0)$ and a polynomial $g(y_{1},y_{2})= \sum_{u,v} c_{u,v} y_{1}^{u} y_{2}^{v}$ with $c_{u,v} \in \mathbb{C}$, where the sum is taken over all pairs $(u,v) \in \mathbb{N}^{2}$ with $u w_{1} + v w_{2}=1$. In this case, the Arnold exponent is given by
$$c_{0}(g) = w_{1}+w_{2}.$$
It is easy to check, using the equations given at the beginning of this section, that all the simple singularities are quasi-homogeneous.
As an example, the equation for $E_6$, namely $g=y_1^3+y_2^4$, is quasi-homogeneous with respect to the weights 
$$w_1=\frac{1}{3}, \ w_2=\frac{1}{4} \text { and hence } c_{0}(E_6) =c_{0}(g) =\frac{1}{3}+\frac{1}{4}=\frac{7}{12}.$$
It is known that a quasi-homogeneous singularity $g=0$ is simple if and only if
\begin{equation} \label{eqSS} 
c_0(g)> \frac{1}{2},
\end{equation}
see for instance \cite[Corollary 7.45 and inequality 7.46]{RCS}.
For a projective hypersurface $V \, : \, f=0$ of any dimension having only isolated quasi-homogeneous singularities, the minimal degree of an element in ${\rm D}_{0}(f)$ is related to the Arnold exponent $\alpha_{V}$, see \cite[Theorem 9]{DimcaSaito} where we can find a rather technical statement. Here we need only the following special case adjusted to curves, see \cite[Theorem 2.1]{DimSer}.
\begin{theorem}[]
\label{sern}
Let $C \, : \, f = 0$ be a reduced curve of degree $n$ in $\mathbb{P}^{2}_{\mathbb{C}}$ having only quasi-homogeneous singularities. Then $${\rm mdr}(f) \geq \alpha_{C}\cdot n -2.$$
\end{theorem}

After such a preparation we are finally ready to present our proof of Theorem \ref{maxi}.
\begin{proof}
If $C$ is a free curve with exponents $(m-1,m)$, then Theorem \ref{thm1}  a) implies that
$$\tau(C)=\tau(2m,m-1)_{max}=3m(m-1)+1.$$
This equality implies that $C$ is maximizing, using Definition \ref{defmax} and Remark \ref{rktau}.
Conversely, assume now that $C \, : f=0$ is a maximizing curve of even degree $n=2m$ with $m\geq 2$ and 
$$\tau(C) =3\frac{n}{2}\bigg(\frac{n}{2}-1\bigg)+1 = 3m(m-1)+1.$$
We are going to show that $C$ is indeed free.
Using \eqref{eqSS} and Theorem \ref{sern}, we get
$$r= {\rm mdr}(f) \geq \alpha_{C}\cdot 2m -2 >\frac{1}{2} \cdot 2m -2 = m-2.$$
It follows that $r \geq m-1$ and then  Theorem \ref{thm1} implies that
$$\tau(C) \leq \tau(2m,m-1)_{max}=3m(m-1)+1.$$
Indeed, we can check directly that the function $\tau(n,r)_{max}$ is strictly decreasing as a function with respect to $r$ on the interval $[0,n-1]=[0,2m-1]$. Since $C$ is maximizing, we have in fact
$$\tau(C) =3m(m-1)+1= \tau(2m,m-1)_{max},$$
which implies that $r=m-1$ and then by Theorem \ref{thm1} a) the curve $C$ is free, with exponents $(m-1,m)$.
This completes the proof.
\end{proof}

\begin{remark}
\label{rk1}
It follows from the above proof that a simply singular curve $C$ of degree $n=2m$ which is maximizing must have some very singular points.
Indeed, for such a curve $C$ one has
$$m-1 \geq 2m \cdot \alpha_C -2$$
and hence
$$\alpha_C \leq \frac{m+1}{2m}.$$
For instance, if the curve $C$ has only $A_k$ singularities, and since
$$c_0(A_k)=\frac{1}{2} +\frac{1}{k+1}$$
it follows that there are some singularities $A_k$ on $C$ with $k \geq 2m-1$. In particular, there is no maximizing curve having only nodes $A_1$ and cusps $A_2$ as singularities.
Similarly, one can show that there is no arrangement $\mathcal{A} \subset \mathbb{P}^{2}_{\mathbb{C}}$ of $n=2m$ lines with $m\geq 4$ having $A_{1}$ and $D_{4}$ singularities (the only simple singularities for a line arrangement) that is a maximizing curve, see \cite[Corollary 4.6]{DimSer}. This shows that in the search of maximizing curves one can use techniques and results devoted to the study of reduced free curves and/or non-existence of such curves. In \cite{DimcaPokora}, we showed that there does not exists any conic-line arrangement of degree $6$ with $A_{1}$, $A_{3}$, $D_{4}$ singularities that is free, so it means that there is no such maximizing curve (cf. \cite[Section 4]{Yang}). 

\end{remark}

\section{Examples of maximazing curves}
Here we present some examples of maximizing curves found by U. Persson in  \cite{Persson}. In view of Theorem \ref{maxi}, they give new examples of free curves.

\begin{example}[Persson's tri-conical arrangement]
\label{P1}
Consider the following arrangement of conics $\mathcal{C} = \{C_{1},C_{2},C_{3}\} \subset \mathbb{P}^{2}_{\mathbb{C}}$, where
\begin{equation*}
\begin{array}{l}
C_{1} :  x^{2} + y^{2} - z^{2} =0,\\
C_{2} :  2x^{2} + y^{2} + 2xz  =0,\\
C_{3} :  2x^{2} + y^{2} - 2xz  =0.\\
\end{array}
\end{equation*}
This arrangement has exactly $5$ singular points, two $A_{1}$ singularities, one $A_{3}$ singularity, and two $A_{7}$ singularities located at $P_{\pm} = (\pm 1 : 0 : 1)$, see \cite[Proposition 2.1]{Persson}. Therefore, the total Tjurina number of $\mathcal{C}$ is equal to
$$\tau(\mathcal{C}) = 2\cdot 1 + 3\cdot 1 + 7 \cdot 2 = 19.$$
Hence $\mathcal{C}$, regarded as the union of the 3 conics, is maximizing and hence a free curve of degree 6 with exponents $(2,3)$.
The fact that the two $A_7$ singularities and the $A_3$ singularity are all on the line $y=0$ plays no role here, but it is essential for the construction of the curve $\mathcal{C}'$ in Example \ref{P2} and of the curve $C'_7$ in Example \ref{exODD1} below.
\end{example}
\begin{remark}
The above configuration is unique up to the projective equivalence. It has two very visible automorphisms, namely $x \mapsto -x$ and $y \mapsto -y$.
\end{remark}
Persson's tri-conical arrangement can be extended to a degree $8$ curve (which is maximizing and hence free, in two distinct ways, as follows.
\begin{example}
\label{P2}
Let $\mathcal{C} = \{C_{1},C_{2},C_{3}\} \subset \mathbb{P}^{2}_{\mathbb{C}}$ be Persson's tri-conical arrangement of conics from Example \ref{P1}. Consider now the arrangement $\mathcal{C}' = \{\ell_{1},\ell_{2},C_{1},C_{2},C_{3}\}$, where $\ell_{1} \, : y=0$ and $\ell_{2} \, : z=0$. This arrangement has $2$ singularities of type $D_{10}$, one singularity of type $D_{6}$, $2$ singularities of type $D_{4}$, and $3$ singularities of type $A_{1}$, see \cite[3.1.1]{Persson}. Since $$\tau(\mathcal{C}')=2\cdot 10 + 6 + 2\cdot 4 + 3\cdot 1 = 37,$$ 
then $\mathcal{C}'$ is a maximizing curve of degree $8$, and hence a free curve with exponents $(3,4)$.

Another way to get a maximizing octic from the sextic $\mathcal{C}$ is to add a new conic, namely 
$$C_4: 2x^2+y^2-2z^2+i\sqrt 2yz=0,$$
see  \cite[Remark 2.8]{Persson}. The arrangement $\mathcal{C}'' = \{C_{1},C_{2},C_{3},C_4\}$ has $2$ singularities of type $D_{10}$, $2$ singularities of type $D_{6}$, $2$ singularities of type $A_1$ and one singularity of type $A_3$, see \cite[3.1.2]{Persson}. It follows that
$$\tau(\mathcal{C}'')=2\cdot 10 + 2\cdot 6 + 2\cdot 1 + 3 = 37.$$ 
As above $\mathcal{C}''$ is a maximizing curve of degree $8$, and hence a free curve with exponents $(3,4)$.

\end{example}
Now we would like to discuss some arrangements constructed using the Steiner quartic curve which also come from \cite{Persson}. Let us recall that all nodal cubics are projectively equivalent and the dual curve to such a nodal cubic is a singular irreducible quartic curve which has exactly $3$ singular points of type $A_{2}$. Observe that such a quartic curve is unique up to the projective equivalence and it is called the Steiner quartic curve. In computations we  can use the following equation for the Steiner quartic:
$$F(x,y,z) = -\frac{1}{4}y^{2}x^{2}-z^{2}(x^{2}+y^{2}-2xy)+x^{2}yz + y^{2}xz.$$
\begin{example}[Maximizing sextic curve derived from the Steiner quartic curve]
\label{P3}
 Let us consider the arrangement $\mathcal{T}_{6} \subset \mathbb{P}^{2}_{\mathbb{C}}$ which consists of the Steiner quartic and two lines, namely $\ell_{1} \, : x=0$ and $\ell_{2} \, : y=0$, which are the tangents to the Steiner quartic at two of its cusps. This curve has $3$ singular points of type $A_{1}$, one singularity of type $A_{2}$ and $2$ singularities of type $E_{7}$. Since
 $$\tau(\mathcal{T}_{6}) = 3\cdot 1 + 2\cdot 1 + 2\cdot 7 = 19,$$
 then $\mathcal{T}_{6}$ is maximizing in degree $6$, and hence free with exponents $(2,3)$.
\end{example}
\begin{example}[Maximizing octic curve derived from the Steiner quartic curve]
\label{P4}
Consider the arrangement $\mathcal{T}_{8}$ consisting of the Steiner quartic, one bitangent line (we can take line at infinity), and three cuspidal tangents. Such an arrangement has $6$ singular points of type $A_{1}$, $2$ points of type $A_{3}$, one point of type $D_{4}$ (since three cuspidal tangents are concurrent), and $3$ points of type $E_{7}$,  see \cite[3.1.2]{Persson}. Since
$$\tau(\mathcal{T}_{8}) = 6\cdot 1 + 2\cdot 3 + 1\cdot 4 + 3\cdot 7 = 37,$$
then $\mathcal{T}_{8}$ is a maximizing curve of degree $8$, and hence free with exponents $(3,4)$.
\end{example}

\begin{remark}
In \cite{Yang}, Yang provides a complete classification of maximizing sextics with only ${\rm ADE}$ singularities. By Theorem \ref{maxi} all these curves are free. Among those curves, we can detect new examples of irreducible reduced curves with ${\rm ADE}$ singularities and there are altogether $128$ irreducible maximizing sextics. Based on that classification result, \textbf{we have at least $128$ simply singular irreducible free plane sextics}.
\end{remark}

\begin{example}[Maximizing curves of arbitrary even degrees]
\label{P5}
In \cite[Lemma 7.8]{Persson} it is shown that the curve
$$\mathcal C_{2(m+1)}: xy[(x^m+y^m+z^m)^2-4(x^my^m+y^mz^m+z^mx^m)]=0$$
is maximizing, and hence we get a curve of even degree $2(m+1)$ which is free with exponents $(m,m+1)$. Following Persson's analysis in \cite{Persson}, notice that $\mathcal C_{2(m+1)}$ has exactly $2m$ singular points of type $D_{m+2}$, exactly $m$ points of type $A_{m-1}$, and one $A_{1}$ point. Based on that
$$\tau(\mathcal C_{2(m+1)}) = 2m\cdot(m+2) + m\cdot(m-1) + 1 = 3m(m+1) + 1,$$ 
so $\mathcal{C}'$ is indeed a maximizing curve.
\end{example}

\begin{remark}
\label{rkNF} The curve $\mathcal C_{2(m+1)}$ in Example \ref{P5} consists clearly of the two lines $x=0$ and $y=0$, and the curve
$$\mathcal D_{2m}: (x^m+y^m+z^m)^2-4(x^my^m+y^mz^m+z^mx^m)=0.$$

This curve is irreducible for $m$ odd, and has exactly 4 smooth components when $m=2k$ is even, namely
\begin{equation*}
\begin{array}{l}
C_{1} :\quad  x^{k} + y^{k} + z^{k} =0,\\
C_{2} :\, -x^k + y^k + z^k =0,\\
C_{3} :\quad  x^k - y^k + z^k =0,\\
C_{4} :\quad x^k + y^k - z^k =0.
\end{array}
\end{equation*}

It is easy to see that $\mathcal D_{2m}$ has $3m$ singularities of type $A_{m-1}$, see \cite[Lemma 7.5]{Persson}. In particular, if $m=4$, then we obtain an arrangement of $4$ smooth conics intersecting along $12$ tacnodes - the maximal possible number of such singular points for $4$ conics. Recall also that an arrangement of $4$ smooth conics and $12$ tacnodes is unique up to the projective equivalence.

This curve $\mathcal D_{2m}$ is not free, but it is nearly free, as proved in \cite[Theorem 3.12]{Artal}.
A nearly free curve is a reduced plane curve $C:f=0$ of degree $n$ such that $r={\rm mdr}(f) < n/2$ and $\tau(C)=\tau(n,r)_{max}-1$, or
$r={\rm mdr}(f) = n/2$ and $\tau(C)=\tau(n,r)_{max}$, see \cite{Dimca1,DimStic} for details, where other characterizations of nearly free curves are also given. In other words, the $\mathcal D_{2m}$ has a Tjurina number which is very close to being maximal.
In the next section we show that these curves 
enjoy some other interesting maximality properties.
\end{remark}

\section{Curves with maximal numbers of $A_k$ singularities}
It turns our that the arrangement $\mathcal{D}_{n}$, with $n=2m$ is extreme in a different sense. In order to explain this phenomenon, we will follow Langer's variation on the Bogomolov-Miyaoka-Yau inequality \cite{Langer} which uses the notion of local orbifold Euler numbers $e_{orb}$ of singular points. Here we will use a specific version of Langer's inequality that is adjusted to our setting of plane curves in the complex projective plane.
\begin{theorem}[Langer]
\label{langer}
Let $C \subset \mathbb{P}^{2}_{\mathbb{C}}$ be a reduced curve of degree $n$ and assume that $(\mathbb{P}^{2}_{\mathbb{C}},\alpha C)$ is a an effective log canonical pair for a suitably chosen $\alpha\in [0,1]$, then one has
$$\sum_{p \in {\rm Sing}(C)} 3\bigg(\alpha(\mu_{p}-1)+1-e_{orb}(p,\mathbb{P}^{2}_{\mathbb{C}},\alpha C)\bigg)\leq (3\alpha - \alpha^{2})n^{2}-3\alpha n,$$
where ${\rm Sing}(C)$ denotes the set of all singular points, $\mu_{p}$ is the Milnor number of a singular point $p$, and $e_{orb}$ denotes the local orbifold Euler number of $p$.
\end{theorem}
In principle, local orbifold Euler numbers are very hard to compute (even though these numbers are analytic in nature), but fortunately these are known for ${\rm ADE}$ singularities due to Langer's work.
Here we only present the crucial steps to get the result, mostly to avoid unnecessary repetition of ideas existing in the literature - the detailed description of each step is given, for instance, in \cite{Pokora2}.

\begin{proposition}
\label{pp1}
Let $C \subset \mathbb{P}^{2}_{\mathbb{C}}$ be a reduced curve of degree $n\geq 6$ with at most ${\rm ADE}$ singularities. Denote by $t_{2k+1}$ the number of $A_{2k+1}$ singularities of $C$ with $k\geq 1$. Then 
$$t_{2k+1} \leq \frac{(k+2)(5k+4)}{12(k^{3}+4k^{2}+4k+1)}n^{2}-\frac{k+2}{2(k^{2}+3k+1)}n,$$

\end{proposition}
\begin{proof}
We will use Theorem \ref{langer} directly. Let us recall that for a singular point $p \in {\rm Sing}(C)$ which is of type $A_{2k+1}$, we have
$$e_{orb}(p,\mathbb{P}^{2}_{\mathbb{C}},\alpha C) =  \frac{(k+2-2(k+1)\alpha)^{2}}{4(k+1)}\,\, \text{ provided that }\,\, \alpha \in \bigg[\frac{k}{2k+2}, \frac{k+2}{2k+2}\bigg].$$
\begin{gather*}
3t_{2k+1}\bigg(2k \alpha + 1 - \frac{(k+2-2(k+1)\alpha)^{2}}{4(k+1)}\bigg) \leq \sum_{p \in {\rm Sing}(C)}3\bigg(\alpha(\mu_{p}-1)+1 - e_{orb}(p, \mathbb{P}^{2}_{\mathbb{C}},\alpha C)\bigg) \\ \leq (3\alpha - \alpha^{2})n^{2} - 3\alpha n.
\end{gather*}
Now our claim follows from the above inequality (after a simple calculation) and by taking $\alpha=\frac{k+2}{2(k+1)}$, since for such a selection of $\alpha$ we are both effective and log canonical.
\end{proof}
\begin{example}
\label{exP1}
Consider the curve $\mathcal{D}_{12}$ from Remark \ref{rkNF}. According to the numerics presented there, we obtain a reduced plane curve of degree $12$ having $18$ singular points of type $A_{5}$. Using the upper-bound proved in Proposition 2.3, we see that 
$$t_{5} \leq \frac{14}{99}n^{2} - \frac{2}{11}n.$$
Taking $n=12$ we obtain
$$t_{5} \leq  \frac{200}{11} \approx 18.18,$$
so $\mathcal{D}_{12}$ maximizes the number of $A_{5}$ singularities in degree $n=12$.
\end{example}
\begin{example}
\label{exP2}
Consider the curve $\mathcal{D}_{16}$ from Remark \ref{rkNF}. According to the numerics presented there, we obtain a reduced plane curve of degree $16$ having $24$ singular points of type $A_{7}$. Using the upper-bound proved in Proposition 2.3, we see that 
$$t_{7} \leq \frac{5}{48}n^{2} - \frac{5}{38}n.$$
Taking $n=16$ we obtain
$$t_{7} \leq \frac{1400}{57} \approx 24.561,$$
so  $\mathcal{D}_{16}$ maximizes the number of $A_{7}$ singularties in degree $n=16$.
\end{example}
\begin{remark}
Taking the proof almost verbatim from Proposition \ref{pp1}, we can get an estimate of the number of $E_{6}$ singularities for a reduced plane curve $C$ of degree $d\geq 6$ with at worst ${\rm ADE}$ singularities. We need to apply Theorem \ref{langer} to the pair $(\mathbb{P}^{2}_{\mathbb{C}},\frac{1}{2}C)$ and for $p \in {\rm Sing}(C)$ which is of type $E_{6}$ we take $$e_{orb}\bigg(p,\mathbb{P}^{2}_{\mathbb{C}}, \frac{1}{2} C\bigg) = \frac{1}{48}.$$ 
If we denote by $e_{6}$ the number of $E_{6}$ singularities of $C$, then
$$e_{6}\leq \frac{20}{167}d^{2}-\frac{24}{167}d.$$
In particular, for $d=18$ we arrive at $$e_{6} \leq \frac{6048}{167} \approx 36.21557.$$ It turns our that there exists a reduced curve $C_{18}$ of degree $18$ with exactly $36$ singularities of type $E_{6}$ constructed by Bonnaf\'e \cite[Example 3.4]{Bonnafe}, which means that $C_{18}$ maximizes the number of $E_{6}$ singularities in degree $18$. Moreover, $C_{18}$ is nearly-free, and this can be checked by a direct computation.
\end{remark}

\section{On the maximal Tjurina number of simply singular curves of odd degree}

Let $C \subset \mathbb{P}^{2}_{\mathbb{C}}$ be a reduced curve of degree $n=2m+1\geq 5 $ with at most ${\rm ADE}$ singularities. Then the result analog to Theorem \ref{maxi} is the following.

\begin{proposition}
\label{propODD1}
Let $C: f=0$ be a reduced curve of degree $n=2m+1\geq 5$ with at most ${\rm ADE}$ singularities. Denote by $r=mdr(f)$ the minimal degree of a derivation in $D_0(f)$ and by $\tau(C)$ the total Tjurina number. Then one of the following two situations occurs.
\begin{enumerate}
\item[a)] $r =m-1$ and  $\tau(C) \leq \tau(n,r)_{max}= 3m^2+1$.

\item[b)] $r =m$ and  $\tau(C) \leq \tau(n,r)_{max}= 3m^2$.

\item[c)] $r >m$ and  $\tau(C) < 3m^2-1$.

\end{enumerate}
In both cases $a)$ and $b)$, the equality holds if and only if the curve $C$ is free.
\end{proposition}
\proof
Theorem \ref{sern} implies that $r \geq m-1$, exactly as in the proof of
Theorem \ref{maxi}. The three cases follow from Theorem \ref{thm1}.
\endproof

In view of the above result, we propose the following.

\begin{definition}
\label{defmaxODD}
A reduced simply singular curve $C \subset \mathbb{P}^{2}_{\mathbb{C}}$ of odd degree $n=2m+1\geq 5$ is a \textbf{maximizing curve} if
$$\tau(C) = 3m^2+1.$$
\end{definition}

\begin{remark}
\label{rkRODD}
Let $k>0$ be a divisor of $n=2m+1=\deg C$,   and let $X$ be the cyclic covering of $\mathbb{P}^{2}_{\mathbb{C}}$
of order $k$ ramified along $C$. If $k \geq 3$ and if we impose the condition that $X$ has only ${\rm ADE}$ singularities, then only the following cases are possible.
\begin{enumerate}
\item[a)] $k =3$ and  $C$ has singularities of type $A_j$ for $1\leq j \leq 4$. The corresponding singularities of $X$ have type $A_2$, $D_4$, $E_6$ and $E_8$.

\item[b)] $k =5$ and  $C$ has singularities of type $A_1$ and $A_2$. The corresponding singularities of $X$ have type $A_4$ and $E_8$.

\item[c)] $k \geq 7$, $k$ odd and  $C$ has singularities of type $A_1$.
The corresponding singularities of $X$ have type $A_{k-1}$.
\end{enumerate}
Using Theorem \ref{sern}, it is easy to see that in all these cases the curve $C$ cannot be free, except when $n=k=3$ and $C:f=xyz=0$ is a triangle. However, this cubic curve $C$ is not maximizing, according to our Definition \ref{defmaxODD}. Therefore, odd degree maximizing curves cannot yield interesting surfaces by taking cyclic covers of $\mathbb{P}^{2}_{\mathbb{C}}$ along them and using the same approach as in the even degree case.
\end{remark}

\bigskip

\begin{question}
\label{qmaxODD} Do maximizing curves $C_{2m+1}$ exist in any degree $n=2m+1\geq 5$ ?
\end{question}

\bigskip

Before dealing with the existence of maximizing curves of any odd degree, we show that the equality in case $b)$ can occur for  curves of any degree $n=2m+1\geq 5$.

\begin{corollary}
\label{corODD1}
For any integer $m \geq 2$, the curve
$$\mathcal C_{2m+1}: f=x[(x^m+y^m+z^m)^2-4(x^my^m+y^mz^m+z^mx^m)]=0$$
is free with exponents $(m,m)$ and satisfies the equality in Proposition \ref{propODD1} in case $b)$.
\end{corollary}

\proof

The curve  $\mathcal C_{2m+1}$ has exactly $m$ singular points of type $D_{m+2}$ and $2m$ points of type $A_{m-1}$. Based on that
$$\tau(\mathcal C_{2m+1}) = m\cdot(m+2) + 2m\cdot(m-1)  = 3m^2.$$ 
It remains to show that $r:={\rm mdr}(f)$ satisfies $r=m$.
If we set
$$f'=(x^m+y^m+z^m)^2-4(x^my^m+y^mz^m+z^mx^m),$$
then it was shown in  \cite[Theorem 3.12]{Artal} that $r'={\rm mdr}(f')=m$.
Then using \cite[Proposition 3.1]{DIS}, it follows that $m \leq r \leq m+1$.
The case $r=m+1$ is excluded using Theorem \ref{thm1} case $b)$.
Therefore $r=m$ and the curve $\mathcal C_{2m+1}$ is free with exponents $(m,m)$.
\endproof

Coming back to the Question \ref{qmaxODD}, we show first that maximizing curves exist in degree $n=5$.

\begin{corollary}
\label{corODD2}
There are exactly 4 free quintics with exponents $(1,3)$ satisfying the equality in Proposition \ref{propODD1} in case $a)$, namely

\[
\begin{array}{ll}
({\rm H}1): \quad E_7+A_5+A_1&   :\,\, f=xz(y^3-xz^2),\\
({\rm H}2): \quad  E_6+A_7&  :\,\, f=z(y^4-x^3z),\\
({\rm H}3): \quad  D_8+D_5&  :\,\, f=yz(y^3-x^2z),\\
({\rm H}4): \quad  2D_6+A_1&  :\,\, f=xyz(y^2-xz).
\end{array}\]
\end{corollary}
 Here we have listed first in each case the singularity types present on the curve. 
 
 \proof This claim follows from Wall's paper \cite{Wall}, beginning of Section 5. Indeed, Wall shows that these 4 quintics are the only quintics $C$ having only ADE singularities and such that $\tau(C)=13$. Since in this case $m=2$, we have just to apply our Proposition \ref{propODD1} to end the proof of our claim.
 \endproof

Note that Wall's paper \cite{Wall} lists all the simply singular quintics $C$ with $\tau(C) \geq 12$. These curves give examples of either nearly free curves in situation $a)$, and hence $\tau(C)= \tau(n,r)_{max}-1$, or of free curves in the situation $b)$.

To construct maximizing curves in odd degrees $\geq 7$ one may try to use the following approach. Let $C_1 : f_1=0$ be a  reduced curve and let $C_2$ be a smooth curve, not an irreducible component of $C_1$, such that the union curve $C = C_1 \cup C_2$ has only ${\rm ADE}$ singularities along $C_2$. For $p \in C_1 \cap C_2$, denote by $i(C_1,C_2)_p$ the intersection multiplicity of the two curves $C_1$ and $C_2$ at $p$. Recall that by B\'ezout Theorem one has
\begin{equation}
\label{eqBEZ}
\sum_p i(C_1,C_2)_p= \deg C_1 \cdot \deg C_2,
\end{equation}
where the sum is over all intersection points $p \in C_1 \cap C_2$.
The following result describes the possible singularities of $C$ along the curve $C_2$.

\begin{proposition}
\label{propODD2}
Let $p \in C_1 \cap C_2$ be a singular point of type $W \in \{A,D,E\}$ for the curve $C=C_1 \cup C_2$. 
We denote by
$$\Delta \tau(C,p)=\tau(C,p)-\tau(C_1,p)$$
the variation of Tjurina number at $p$ when passing from $C_1$ to $C$, and by $i(C_1,C_2)_p$ the intersection multiplicity of $C_1$ and $C_2$ at $p$.
Then the following cases are possible:

\begin{enumerate}
\item[1)] $p$ is a smooth point on $C_1$ and
the smooth curve $C_2$ is tangent of order $k\geq 0$ to $C_1$ at the point $p$.
Then $(C,p)$ is a singularity of type $A_{2k+1}$ and
$$\Delta \tau(C,p)- \frac{3}{2}i(C_1,C_2)_p=\frac{k-1}{2}.$$
\item[2)] $p$ is a singular point of type $A_k$ on $C_1$ for some $k\geq 1$ and the tangent at $p$ of the smooth curve $C_2$ is not in the tangent cone of the singularity $(C_1,p)$. Then
$(C,p)$ is a singularity of type $D_{k+3}$ and
$$\Delta \tau(C,p)- \frac{3}{2}i(C_1,C_2)_p=0.$$
\item[3)] $p$ is a singular point of type $A_1$ on $C_1$ and the smooth curve $C_2$ is tangent of order $k\geq 1$ to one of the two smooth branches. 
 Then
$(C,p)$ is a singularity of type $D_{2k+4}$ and
$$\Delta \tau(C,p)- \frac{3}{2}i(C_1,C_2)_p=\frac{k}{2}.$$
\item[4)] $p$ is a singular point of type $A_2$ on $C_1$ and the tangent at $p$ of the smooth curve $C_2$ is the tangent cone to this cusp. Then
$(C,p)$ is a singularity of type $E_7$ and
$$\Delta \tau(C,p)- \frac{3}{2}i(C_1,C_2)_p=\frac{1}{2}.$$

\end{enumerate}

\end{proposition}

\proof
Since a simple singularity has multiplicity at most $3$ it is clear that $p$ must be either a smooth point or an $A_k$-singularity for $C_1$. Note that the singularities $A_j$ with $j$-even and $E_6$, $E_8$, being irreducible, cannot occur as singularities $(C,p)$.

In Case 1, $p$ is a smooth point on $C_1$ and one can find local coordinates $(u,v)$ centered at $p$ such that
$C_1: u+v^{k+1}=0$, $C_2:u=0$. It follows that the intersection multiplicity
$i(C_1,C_2)_p$ is equal to $k+1$ and the germ $(C,p)$ is given by $u(u+v^{k+1})=0$, that is $(C,p)$ is a singularity of type $A_{2k+1}$. 

In Case 2,  $p$ is a singular point of type $A_k$ on $C_1$ and one can find local coordinates $(u,v)$ centered at $p$ such that
$C_1: u^2+v^{k+1}=0$, $C_2:v=0$. It follows that the intersection multiplicity
$i(C_1,C_2)_p$ is equal to $2$ and the germ $(C,p)$ is given by $v(u^2+v^{k+1})=0$, that is $(C,p)$ is a singularity of type $D_{k+3}$.

In Case 3, the resulting singularity of $(C,p)$ is of type $D_{2k+4}$, since this situation coincides with the case of an $A_{2k+1}$-singularity and a transversal curve $C_2$ discussed in the case 2. In this case one also has $i(C_1,C_2)_p=k+2.$

Finally, in Case 4, $(C_1,p)$ is an $A_2$ singularity given by $u^2-v^3=0$ and $C_2$ is tangent to the line $u=0$. Then the resulting singularity has type $E_7$, and moreover
$i(C_1,C_2)_p=3$.

All the other cases for the pair $(C_1,p)$ and $C_2$ give rise to non-simple singularities.
\endproof
We see that a singularity of type $D_{k}$ can occur on the curve $C$ in two different ways, either as in Case 2 above, or as in Case 3 above.
If a singularity $D_j$ comes as in Case 2, we call it {\it transversal} and denote it by $D_j^t$. On the other hand, if a singularity $D_j$ comes as in Case 3, we call it {\it non-transversal} and denote it by $D_j^n$. Let
$N(X)$ denotes the number of singularities of type $X$ that the curve $C$ has along the line $L$. With this notation we have the following result, obtained when $C_2$ is a line $L$.

\begin{theorem}
\label{thmODD}
Let $C_1:f_1=0$ be a curve of even degree $2m\geq 4$ having only ${\rm ADE}$ singularities
such that
$$\tau(C_1)=3m(m-1)+1-\delta,$$
for some integer $\delta \geq 0$. Let $L$ be a line such that the  curve $C =C_1 \cup L$ has only ${\rm ADE}$ singularities and 
\begin{equation}
\label{eqODD}
2\delta+N(A_1) \leq \sum_{j >1}(j-1)N(A_{2j+1})+\sum_{j >0}jN(D_{2j+4}^n)+N(E_7).
\end{equation}
Then the curve $C$ is maximizing of degree $2m+1$ and equality holds in \eqref{eqODD}.
\end{theorem}

\proof
Using Proposition \ref{propODD1}, it is enough to show that $\tau(C)=3m^2+1$. It follows that it is enough to show that
$$\sum_p \Delta \tau(C,p) = 3m+\delta,$$
where the sum is over all points $p \in C_1 \cap L$. Using now Proposition \ref{propODD2} and the equality
$$\sum_{p} i(C_1,L)_p=\deg C_1=2m,$$
coming from \eqref{eqBEZ},
we get
$$\sum_p \Delta \tau(C,p)-3m=
 \sum_p( \Delta \tau(C,p) - \frac{3}{2}i(C_1,L)_p)=$$
 
$$= \frac{-N(A_1)+\sum_{j >1}(j-1)N(A_{2j+1})+\sum_{j >0}jN(D_{2j+4}^n)+N(E_7)} {2}\geq \delta.$$
Since $\tau(C)$ cannot be greater that $3m^2+1$, it follows that we have
equality on the last line above, and hence equality holds in  \eqref{eqODD} as well.
This ends the proof that $C$ is a maximizing curve.
\endproof

\begin{example}
\label{exODD0}
In this Example we construct first a  maximizing sextic, starting with the smooth Fermat cubic
$$C': x^3+y^3+z^3=0.$$
 Consider the secant line $L:z=0$, with meets the cubic
 $C'$ in 3 inflection points, call them $p_1$, $p_2$ and $p_3$.
 Let $L_j$ be the tangent to $C'$ at the inflection point $p_j$ for $j=1,2,3$. Consider the sextic
 $$C_1=C' \cup L_1 \cup L_2 \cup L_3.$$
 Then $C_1$ has $3$ singularities of type $A_5$ at the points $p_j$ for $j \in  \{1,2,3\}$  and one $D_4$-singularity at $q=(0:0:1)$, the common intersection point of the $3$ lines $L_j$. Hence 
 $$\tau(C_1) = 3 \cdot 5 + 4=19$$
 and $C_1$ is a maximizing sextic. To get a maximizing septic, we add the line $L$ and apply Theorem \ref{thmODD} with $\delta =0$.
 
Note that for a general smooth cubic
$$C_t': x^3+y^3+z^3-3txyz=0$$ 
with $t \ne 0$ and $t^3 \ne 1$, the corresponding 3 lines $L_j$ are no longer concurrent and hence the sextic  $C_1=C' \cup L_1 \cup L_2 \cup L_3$ is not maximizing. This brings new evidence to Hirzebruch's conjecture that the maximizing sextics modulo projective equivalence are in finite number, see \cite[bottom of p. 67]{Hirz82}. This question does not seem to be discussed in \cite{Yang}, where the list of possible singularities of maximal sextics is given.

 A similar construction starts with the nodal cubic
 $$C': xyz+x^3+y^3=0$$
 and take $L, L_1,L_2,L_3$ as above. Then the 3 lines $L_j$ are no longer concurrent, and hence instead of a $D_4$ singularity they produce 3 nodes $A_1$. However, the sextic $C_1=C' \cup L_1 \cup L_2 \cup L_3$ is again maximizing, due to the fourth node at $q=(0:0:1)$,
 and the septic $C=C_1 \cup L$ is also maximizing as above.

\end{example}

\begin{example}
\label{exODD1}
In this Example we construct two maximizing curves of degree 7 using Theorem \ref{thmODD}.
The first one, call it $C_7$,  is  obtained from Steiner quartic curve by adding all the three tangents at the singular points. With the notation from Example \ref{P3}, $C_7$ is obtained from the sextic $\mathcal{T}_{6}$ by adding the line $L: z=0$. This sextic is maximizing, hence $\delta=0$ in this case.
The 3 singularities on the line $L$ are
of types $A_1$, $D_4^t$ and $E_7$ on the curve $C_7$. A nice picture of this curve is given in \cite[p. 66]{Hirz82}.

The second curve, call it $C'_7$, is obtained from the curve 
$\mathcal{C}$ in Example \ref{P1} above by adding the line $L:y=0$. 
The 3 singularities on the line $L$ are
of types $D_6^t$, $D_{10}^t$ and $D_{10}^t$ on the curve $C'_7$.
This sextic $\mathcal{C}$  is maximizing, hence again $\delta=0$.

\end{example}
\begin{remark}
\label{rkODD0}
(i) It follows from \cite[Corollary 6.6]{DIS} that if $C_1$ is a maximizing curve of degree $2m$ and $L$ is any line such that $C=C_1 \cup L:f=0$ has only quasi-homogeneous singularities, then $|C_1 \cap L| \geq m$.
Moreover, the equality $|C_1 \cap L| =m$ implies $r=mdr(f)=m-1$, hence one of the conditions a maximizing curve $C$ has to fulfill.

(ii) The condition \eqref{eqODD} in Theorem \ref{thmODD} is necessary. Indeed, the curve of degree $2m+3$ obtained from the curve $C_1=\mathcal C _{2m+2}$ in Example \ref{P5} by adding the line $L:z=0$ is not maximizing.
The singularities of $C=C_1 \cup L$ along the line $L$ are as follows: two $A_1$ singularities and $m$ singularities of type $D_{m+2}^t$.

\end{remark}

Using the same proof as for Theorem  \ref{thmODD}, one can prove the following  result, obtained when $C_2$ is a line $L$ and $C_1$ has even degree.

\begin{theorem}
\label{thmEVEN}
Let $C_1:f_1=0$ be a reduced curve of odd degree $2m+1\geq 3$ having only ${\rm ADE}$ singularities
such that
$$\tau(C_1)=3m^2+1-\delta,$$
for some integer $\delta \geq 0$.
Let $L$ be a line such that the  curve $C =C_1 \cup L$ has only ${\rm ADE}$ singularities and 
\begin{equation}
\label{eqEVEN}
2 \delta +N(A_1) \leq 3+\sum_{j >1}(j-1)N(A_{2j+1})+\sum_{j >0}jN(D_{2j+4}^n)+N(E_7).
\end{equation}
Then the curve $C$ is maximizing of degree $2m+2$ and equality holds in  \eqref{eqEVEN}.
\end{theorem}

\begin{example}
\label{exEVEN0}
In this Example, we construct {\it five new examples} of maximizing curves of degree 6 using Corollary \ref{corODD2} and Theorem \ref{thmEVEN}.
The first three examples are constructed by adding a tangent line.
Consider one of the maximizing curves of types $(H2)$, $(H3)$ or $(H4)$ in Corollary \ref{corODD2}, call it $C_1$ and add a new simple tangent $L$ to the irreducible component of $C_1$ of degree $>1$, distinct from the existing tangents if any.
Then along this line $L$, the curve $C=C_1 \cup L$ has one $A_3$ singularity and 3 singularities $A_1$. Using Theorem \ref{thmEVEN} for $\delta=0$ we get that the sextic $C$ is maximizing.

The last two examples are obtained by adding a line passing through a double point. Consider one of the maximizing curves of types $(H1)$ or $(H4)$ in Corollary \ref{corODD2}, call it $C_1$ and add a generic line $L$ passing through the double point $p=(0:1:0)$. Then along this line $L$, the curve $C=C_1 \cup L$ has one $D_4^t$ singularity and 3 singularities $A_1$. Using Theorem \ref{thmEVEN} for $\delta=0$ we get again that the sextic $C$ is maximizing.
\end{example}

\begin{example}
\label{exEVEN1}
In this Example, we construct first a maximizing curve of degree 8 using Theorem \ref{thmEVEN}.
If $C_7$  is  the curve obtained from Steiner quartic curve by adding all the three tangents, and $L$ is a bitangent as in Example \ref{P4}, then
the singularities of $C=C_7 \cup L$ on the line $L$ are
of types $3A_1$, and $2A_3$. Theorem \ref{thmEVEN} implies that $C$ is a maximizing octic, a fact already remarked in Example \ref{P4} above.

Next we regard the maximizing curve $\mathcal C _{2m+2}$ from Example \ref{P5} as obtained from the curve $\mathcal C _{2m+1}$
from Corollary \ref{corODD1} by adding the line $L:y=0$. Note that
$\delta=1$ in this case and the singularities along $L$ are of types
$A_1$ and $m$ times $D_{m+2}^t$. Hence we have equality in \eqref{eqEVEN} as expected, and this gives a new proof for the claim that $\mathcal C_{2m+2}$ is maximizing.

\end{example}

Using the same proof as for Theorem  \ref{thmODD}, one can prove the following  result, obtained when $C_2$ is a smooth conic $Q$ and $C_1$ has even degree.

\begin{theorem}
\label{thmEVEN2}
Let $C_1:f_1=0$ be a reduced curve of even degree $2m \geq 4$ having only ${\rm ADE}$ singularities
such that
$$\tau(C_1)=3m(m-1)+1-\delta,$$
for some integer $\delta \geq 0$.
Let $Q$ be a smooth conic such that the  curve $C = C_1 \cup Q$ has only ${\rm ADE}$ singularities and 
\begin{equation}
\label{eqEVEN2}
2 \delta +N(A_1) \leq \sum_{j >1}(j-1)N(A_{2j+1})+\sum_{j >0}jN(D_{2j+4}^n)+N(E_7).
\end{equation}
Then the curve $C$ is maximizing of degree $2m+2$ and equality holds in  \eqref{eqEVEN2}.
\end{theorem}
\proof
Using Proposition \ref{propODD1}, it is enough to show that $\tau(C)=3m(m+1)+1$. It follows that it is enough to show that
$$\sum_p \Delta \tau(C,p) = 6m+\delta,$$
where the sum is over all points $p \in C_1 \cap Q$. Using now Proposition \ref{propODD2} and the equality
$$\sum_{p}i(C_1,Q)_p=2 \deg C_1=4m,$$
coming from \eqref{eqBEZ},
we get
$$\sum_p \Delta \tau(C,p) - 6m=
 \sum_p( \Delta \tau(C,p) - \frac{3}{2}i(C_1,Q)_p)=$$
 $$= \frac{-N(A_1)+\sum_{j >1}(j-1)N(A_{2j+1})+\sum_{j >0}jN(D_{2j+4}^n)+N(E_7)} {2}\geq \delta.$$
Since $\tau(C)$ cannot be greater that $3m(m+1)+1$, it follows that we have
equality on the last line above, and hence equality holds in  \eqref{eqEVEN2} as well.
This ends the proof that $C$ is a maximizing curve.
\endproof

\begin{example}
\label{exEVEN2}
In this Example, we construct first a maximizing curve of degree $8$ using Theorem \ref{thmEVEN2}. This curve occurs already in \cite[p. 65]{Hirz82}.
The curve $C_1$ is the following real line arrangement: start with an equilateral triangle and add the 3 lines joining one vertex to the middle point of the opposite side. This is a maximal sextics, since it has $3$ nodes and $4$ triple points, hence
$$\tau(C_1)=3+ 4 \cdot 4=19.$$
Take the conic $Q$ to be the circle passing through the $3$ middle points of the $3$ sides of the triangle, with center the barycenter of this 
equilateral triangle. Then along this conic $Q$, the curve $C=C_1 \cup Q$ has three $D_6^n$ singularities and $3$ singularities $A_1$. Using Theorem \ref{thmEVEN2} for $\delta=0$ we get again that the octic $C$ is maximizing.

The reader can check that the construction of the maximal octic $C=\mathcal C''$ from Example \ref{P2}, which is obtained from the sextic $C_1=\mathcal C$ by adding the smooth conic $Q=C_4$ can also be seen as an illustration of Theorem \ref{thmEVEN2} for $\delta=0$.
In this case, along this conic $Q$, the curve $C$ has two $D_{10}^t$ singularities, 2 singularities $A_1$ and two $D_{6}^n$ singularities.

\end{example}
There is of course a result similar to Theorem \ref{thmEVEN2} where we start with an odd degree curve $C_1$ and add a smooth conic.

\begin{theorem}
\label{thmODD2}
Let $C_1:f_1=0$ be a reduced curve of odd degree $2m+1 \geq 5$ having only ${\rm ADE}$ singularities
such that
$$\tau(C_1)=3m^2+1-\delta,$$
for some integer $\delta \geq 0$.
Let $Q$ be a smooth conic such that the  curve $C =C_1 \cup Q$ has only ${\rm ADE}$ singularities and 
\begin{equation}
\label{eqODD2}
2 \delta +N(A_1) \leq \sum_{j >1}(j-1)N(A_{2j+1})+\sum_{j >0}jN(D_{2j+4}^n)+N(E_7).
\end{equation}
Then the curve $C$ is maximizing of degree $2m+3$ and equality holds in \eqref{eqODD2}.
\end{theorem}
The proof is exactly the same as for Theorem \ref{thmEVEN2}.
However, we have no clear application of Theorem \ref{thmODD2} for the time being.

\begin{remark}
\label{rkODD1}
It is easy to see that there are no line arrangements of odd degree $n=2m+1\geq 5$ which are maximizing curves.
Indeed, a line arrangement $C$ is simply singular if and only if it has only $A_1$ and $D_4$ singularities, hence in the notation from Theorem \ref{sern} we get $\alpha_C \geq 2/3$. Assuming that $C$ is maximizing, would give
$$m-1 \geq \frac{2(2m+1)}{3}-2,$$
which implies $m\leq 1$. The only possibility is $m=1$, and the line arrangement $C:x^3-y^3=0$.

One can list the line arrangements of odd degree $n=2m+1\geq 5$ satisfying the equality in Proposition \ref{propODD1} in case $b)$. Using the same approach as above, we get $m\leq 4$. The case 
$m=2$ yields the arrangement 
$$\mathcal{A}_5: xy(x-z)(y-z)(x-y)=0,$$
the case $m=3$ yields 
the arrangement 
$$\mathcal{A}_7: z(x^2-z^2)(y^2-z^2)(x^2-y^2)=0,$$
and finally the case $m=4$ yields the arrangement
$$\mathcal{A}_9: (x^3-z^3)(y^3-z^3)(x^3-y^3)=0.$$
The reader may check easily that these 3 line arrangements are indeed
satisfying the equality in Proposition \ref{propODD1} in case $b)$.

\end{remark}

\section*{Acknowledgments}

The second author would like to thank  Marco Golla for discussions regarding the content of the paper and for useful remarks. We would like to thank the anonymous referees for their very careful reading of our manuscript and for their suggestions to improve the presentation.

Alexandru Dimca was partially supported by the Romanian Ministry of Research and Innovation, CNCS - UEFISCDI, Grant \textbf{PN-III-P4-ID-PCE-2020-0029}, within PNCDI III.

Piotr Pokora was partially supported by the National Science Center (Poland) Sonata Grant Nr \textbf{2018/31/D/ST1/00177}.
\section*{Conflict of Interest Statement}
 On behalf of all authors, the corresponding author states that there is no conflict of interest.
\section*{Data availability} Not applicable as the results presented in this manuscript rely on no external sources of data or code.

\vskip 0.5 cm

Alexandru Dimca,
Universit\'e C\^ ote d'Azur, CNRS, LJAD, France and Simion Stoilow Institute of Mathematics,
P.O. Box 1-764, RO-014700 Bucharest, Romania. \\
\nopagebreak
\textit{E-mail address:} \texttt{dimca@unice.fr} \\

Piotr Pokora,
Department of Mathematics,
Pedagogical University of Krakow,
Podchor\c a\.zych 2,
PL-30-084 Krak\'ow, Poland. \\
\nopagebreak
\textit{E-mail address:} \texttt{piotr.pokora@up.krakow.pl}


\begin{thebibliography}{000}

\bibitem{Artal}
E. Artal Bartolo, L Gorrochategui, I. Luengo, and A. Melle-Hern\'andez, On some conjectures about free and nearly free divisors. Decker, Wolfram (ed.) et al., Singularities and computer algebra. Festschrift for Gert-Martin Greuel on the occasion of his 70th birthday. Based on the conference, Lambrecht (Pfalz), Germany, June 2015. Cham: Springer (ISBN 978-3-319-28828-4/hbk; 978-3-319-28829-1/ebook). 1-19 (2017).

\bibitem{Bar}
M. Barakat, R. Behrends, Ch. Jefferson, L. K\" uhne, M. Leuner, On the Generation of Rank 3 Simple Matroids with an Application to Terao's Freeness Conjecture. \textit{SIAM J. Discret. Math.} \textbf{35}: 1201 -- 1223 (2021).

\bibitem{Beau}
A. Beauville,
Some surfaces with maximal Picard number, \textit{J. Éc. Polytech., Math.} \textbf{1}: 101 -- 116 (2014).
\bibitem{Bonnafe}
C. Bonnaf\'e, Some singular curves and surfaces arising from invariants of complex reflection groups. \textit{Exp. Math.} \textbf{30(3)}: 429 -- 440 (2021); correction ibid. \textbf{29(3)}: 360 (2020).
\bibitem{Singular}
W.~Decker, G.-M. Greuel, G.~Pfister, and H.~Sch\"onemann,
\newblock {\sc Singular} {4-1-1} --- {A} computer algebra system for polynomial computations. \newblock \url{http://www.singular.uni-kl.de}, 2018.

\bibitem{RCS}  A. Dimca,  Topics on Real and Complex Singularities, Vieweg Advanced Lecture in 
Mathematics, Friedr. Vieweg und Sohn, Braunschweig, 1987.


\bibitem{Dimca1}
A. Dimca, Freeness versus Maximal Global Tjurina Number for Plane Curves. \textit{Math. Proc. Camb. Philos. Soc.} \textbf{163(1)}: 161 -- 172 (2017).

\bibitem{DIS}
A. Dimca, G. Ilardi, G. Sticlaru, Addition-deletion results for the minimal degree of a Jacobian syzygy of a union of two curves. \textit{J. Algebra} \textbf{615(1)}: 77 -- 102 (2023).

\bibitem{DJP}
A. Dimca, M. Janasz, P. Pokora, On plane conic arrangements with nodes and tacnodes. \textit{Innov. Incidence Geom.} \textbf{19(2)}: 47 -- 58 (2022).
\bibitem{DimcaPokora}
A. Dimca and P. Pokora, On conic-line arrangements with nodes, tacnodes, and ordinary triple points. \textit{J. Algebraic Combin.} \textbf{56(2)}: 403 -- 424 (2022).
\bibitem{DimcaSaito}
A. Dimca and M. Saito, Generalization of theorems of Griffiths and Steenbrink to hypersurfaces with ordinary double points. \textit{Bull. Math. Soc. Sci. Math. Roum., Nouv. S\'er.} \textbf{60(108), No. 4}: 351 -- 371 (2017).
\bibitem{DimSer}
A. Dimca and E. Sernesi, Syzygies and logarithmic vector fields along plane curves. (Syzygies et champs de vecteurs logarithmiques le long de courbes planes.) \textit{J. Éc. Polytech., Math.} \textbf{1}: 247 -- 267 (2014).
 \bibitem{DimStic} A. Dimca and G. Sticlaru, Free and Nearly Free Curves vs. Rational Cuspidal Plane Curves. \textit{Publ. Res. Inst. Math. Sci.} \textbf{54(1)}: 163 -- 179 (2018). 
 
 \bibitem{DStmax} A. Dimca, G. Sticlaru, Jacobian syzygies, Fitting ideals, and plane curves with maximal global Tjurina numbers. \textit{Collect. Math.} \textbf{73}: 391 -- 409 (2022).
 
 
\bibitem{duP}
A. A. Du Plessis and C. T. C. Wall, Application of the theory of the discriminant to highly singular plane curves. \textit{Math. Proc. Camb. Philos. Soc.} \textbf{126(2)}: 259 -- 266 (1999).


\bibitem{HS12} S. H. Hassanzadeh and A. Simis, Plane Cremona maps: Saturation and regularity of the base ideal. \textit{J. Algebra} {\bf 371}:  620 -- 652 (2012).

\bibitem{Hirz82} F. Hirzebruch, Some examples of algebraic surfaces. \textit{Contemp. Math}. \textbf{9}: 55 -- 71 (1982).

\bibitem{Hirzebruch}
F. Hirzebruch, Singularities of algebraic surfaces and characteristic numbers. Algebraic geometry, Proc. Lefschetz Centen. Conf., Mexico City/Mex. 1984, Part I, Contemp. Math. \textbf{58}: 141 -- 155 (1986).

\bibitem{Langer}
A. Langer, Logarithmic orbifold Euler numbers with applications. \textit{Proc. London Math. Soc.} \textbf{86}: 358 -- 396 (2003).

\bibitem{Persson}
U. Persson, Horikawa surfaces with maximal Picard numbers. \textit{Math. Ann.} \textbf{259}: 287 -- 312 (1982).

\bibitem{Pokora2}
P. Pokora. \newblock The orbifold Langer-Miyaoka-Yau inequality and Hirzebruch-type inequalities. \newblock \textit{Electron. Res. Announc. Math. Sci.} \textbf{24}: 21 -- 27 (2017).

\bibitem{PSz}
P. Pokora and T. Szemberg, Conic-line arrangements in the complex projective plane. \textit{Discrete Comput. Geom}. \url{https://doi.org/10.1007/s00454-022-00397-6} (2022).

\bibitem{KS} K. Saito, Theory of logarithmic differential forms and logarithmic vector fields. \textit{J. Fac. Sci. Univ. Tokyo Sect. IA Math.} \textbf{27(2)}: 265 -- 291 (1980).


\bibitem{SchenckToh}
H. Schenck and S. Toh\u aneanu, Freeness of conic-line arrangements in $\mathbb{P}^{2}$. \textit{Comment. Math. Helv.} \textbf{84(2)}: 235 -- 258 (2009).

\bibitem{ST} A. Simis and S. Toh\u aneanu, Homology of homogeneous divisors. \textit{Israel J. Math.} \textbf{200}: 449 -- 487 (2014).

\bibitem{Te} H. Terao, Generalized exponents of a free arrangement of hyperplanes and Shephard-Todd-Brieskorn formula. \textit{Invent. Math.} \textbf{63}: 159 -- 179 (1981).

\bibitem{Wall} C.T.C. Wall, Highly singular quintic curves. \textit{Math. Proc. Cambridge Phil. Soc.} \textbf{119(2)}: 257 -- 277 (1996).


\bibitem{Yang}
J.-G. Yang, Sextic curves with simple singularities. \textit{T\^ohoku Math. J., II. Ser.} \textbf{48(2)}: 203 -- 227 (1996).

\end{thebibliography}
\end{document}